\documentclass{article}
\usepackage{geometry}
\usepackage{titling}
\AtEndDocument{\bigskip{\footnotesize%
\textsc{Department of Applied Mathematics, Delhi Technological University, Delhi–110042, India} \par  
  \textit{E-mail address:} \texttt{suryagiri456@gmail.com} \par
  \addvspace{\medskipamount}
  \textsc{Department of Applied Mathematics, Delhi Technological University, Delhi–110042, India} \par
  \textit{E-mail address:} \texttt{spkumar@dtu.ac.in} \par
}}
\usepackage{doc}

\usepackage{url}

\usepackage{graphicx}
\usepackage{epstopdf}
\usepackage{amsthm}
\usepackage{amssymb}
\usepackage{amsmath}
\usepackage{enumitem}

\usepackage{array}
\usepackage{enumitem}
\usepackage[utf8]{inputenc}
\usepackage[english]{babel}
\newtheorem{theorem}{Theorem}
\newtheorem{corollary}{Corollary}[theorem]
\newtheorem{lemma}[theorem]{Lemma}

\DeclareMathOperator{\RE}{Re}
\DeclareMathOperator{\IM}{Im}

\usepackage{varwidth}
\usepackage{lineno, hyperref}
\usepackage{graphicx}
\usepackage{epstopdf}
\usepackage{amsmath}
\usepackage{amsthm}
\usepackage{enumitem}
\usepackage{amsthm}
\usepackage{float}
\usepackage{subcaption}
\usepackage{caption}
\usepackage{authblk}
\usepackage{abstract}

\begin{document}


\title{Starlike Functions Associated with a Non-Convex Domain}
\author{S. Sivaprasad Kumar and Surya Giri }
\date{}

\maketitle





\begin{abstract}
    \noindent   We introduce and study a class of starlike functions associated with the non-convex domain
\[
\mathcal{S}^*_{nc} = \left\{ f \in \mathcal{A} : \frac{z f'(z)}{f(z)} \prec \frac{1+z}{\cos{z}} =: \varphi_{nc}(z), \;\; z \in \mathbb{D} \right\}.
\]
Key results include the growth and distortion theorems, initial coefficient bounds, and the sharp estimates for third-order Hankel and Hermitian-Toeplitz determinants. We also examine inclusion relations, radius problems for certain subclasses, and subordination results.

These findings enrich the theory of starlike functions associated with non-convex domains, offering new perspectives in geometric function theory.
\end{abstract}
\vspace{0.5cm}





%

\section{Introduction}
   Let $\mathcal{A}$ denote the class of analytic functions $f$ in the open unit disc $\mathbb{D} =\{z\in\mathbb{C}:\vert z\vert  <1\}$ normalized by $f(0) =0,$ $f'(0) =1$, and let $\mathcal{S}\subset\mathcal{A}$ be the subclass of functions that are univalent in $\mathbb{D}$. The well known subclasses of $\mathcal{S}$ are the classes of starlike and convex functions denoted by $\mathcal{S}^*$ and $\mathcal{C}$ respectively. A function $f \in \mathcal{S}^*$  (or $\mathcal{C}$) if and only if it satisfies $ \RE (z f'(z)/f(z)) >0  $ (or $\RE (1 + zf''(z)/f'(z)) > 0$) for $z \in \mathbb{D}$. An analytic function $f$ is said to be subordinate to other analytic function $g$ in $\mathbb{D}$ if there exists a Schwarz function $\omega(z)$ such that $f(z) = g(\omega (z))$, $z\in \mathbb{D}$. It is denoted by $f \prec g$. 
   In 1992, Ma and Minda~\cite{MaMi} considered the class $\mathcal{S}^*(\varphi)$, which is defined as
\begin{equation}\label{mamieq}
     \mathcal{S}^*(\varphi)=\bigg\{ f\in\mathcal{A}: \frac{z f'(z)}{f(z)} \prec \varphi(z) \bigg\},
\end{equation}
    where $\varphi(z)$ is an analytic univalent function in $\mathbb{D}$ such that $\varphi(0)=1$ and maps the unit disk onto the domain satisfying the following properties: (i) $\varphi(\mathbb{D})$ is starlike with respect to $\varphi(0)=1$ (ii) $\RE (\varphi(z))> 0$ (iii) and $\varphi(\mathbb{D})$ is symmetric about the real axis.

     This class covers various subclasses of starlike functions, for instance, when $\varphi(z) = (1 + A z )/ (1 + B z)$ $(-1 \leq B < A \leq 1)$, the class $\mathcal{S}^*(\varphi)$ becomes the class of the Janowski starlike functions~\cite{Janow}. The class $\mathcal{S}^*(\alpha)$ containing starlike functions of order $\alpha$ can be obtained by taking $\varphi(z)=(1 + (1-2 \alpha) z)/(1-z)$ ($0 \leq \alpha < 1$). For $\varphi(z)= 1 + \left(2/{\pi^2}\right) \left(\log\left((1 + \sqrt{z})/(1 - \sqrt{z})\right)\right)^2$, the class $\mathcal{S}^*(\varphi)$ reduces to the class of parabolic starlike functions introduced by R{\o}nning~\cite{Ronn}.
     In case of modified sigmoid function given by $\varphi(z) = 2/(1+ e^{-z})$, we obtain the class $\mathcal{S}^*_{SG}$ introduced and studied by Goel and Kumar~\cite{GoSS}. Geometrically, a function $f \in \mathcal{S}^*_{SG}$ if and only if $z f'(z)/f(z)$ lies in the domain $\Delta_{SG} = \{ w\in \mathbb{C} : \vert \log (w/(2-w))\vert <1 \}. $
       Recently, Kumar and Gangania~\cite{KK} considered $\varphi(z) = 1 + z e^z$, which maps the unit disk onto a cardioid domain, and introduced the subclass $\mathcal{S}^*_{\varrho}$. Numerous subclasses of $\mathcal{S}^*$ have been considered by choosing an appropriate $\varphi$ in (\ref{mamieq}). For instance, the interesting regions depicted by the functions $e^z$, $ 1 + \sin{z}$ and $(1 + z)^s$ were considered in place of $\varphi(z)$ by Mendiratta et al.~\cite{Mendii}, Cho et al.~\cite{ChoKu} and Liu and Liu~\cite{LiuLiu}, respectively. Motivated by these defined classes, we introduce the class of starlike functions associated with a non-convex shaped domain represented by $\varphi_{nc}:= (1+z)/\cos{z}$ (see Figure \ref{image}). Let
    $$\mathcal{S}_{nc}^*:=\bigg\{ f\in\mathcal{A} : \frac{zf'(z)}{f(z)}\prec \frac{1+z}{\cos{z}}\bigg\} . $$
\begin{figure}[h]
\centering
   \includegraphics[width=5.02cm, height=4.02cm]{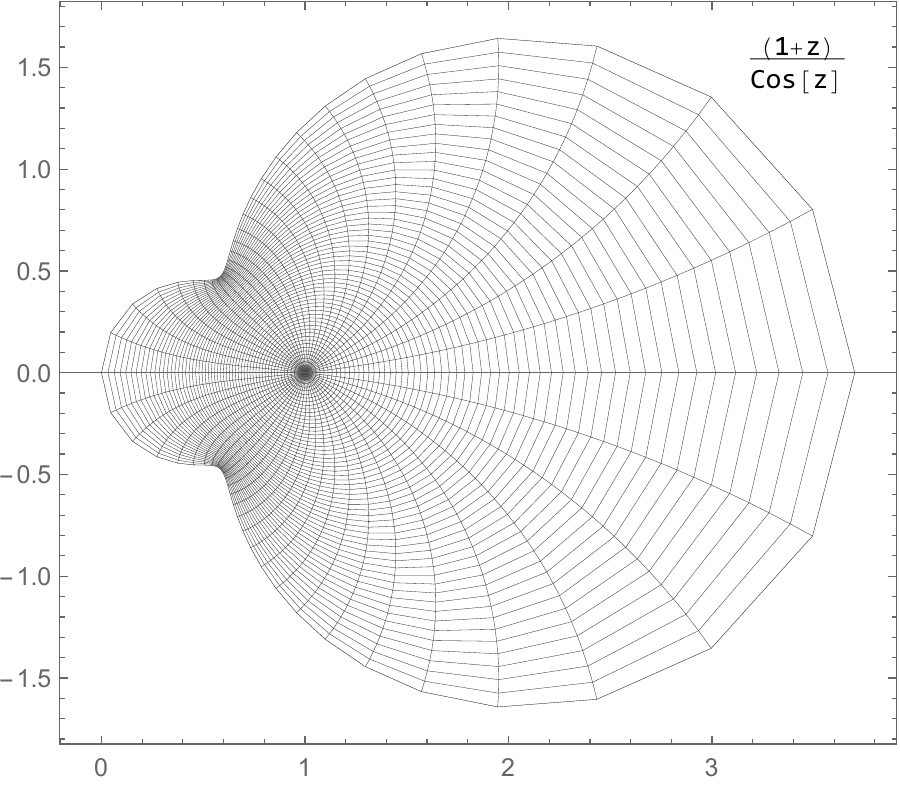}
 \caption{Image of the unit disk under the function $\varphi_{nc}$.}
 \label{image}
\end{figure}
    We study the geometric properties of functions belonging to the class $\mathcal{S}^*_{nc}$. 
     In geometric function theory, Bieberbach conjecture for the class $\mathcal{S}$ was the famous coefficient problem, which took 66 years to solve. In between, the same problem for other subclasses of $\mathcal{S}$ were verified. Finding the bounds of Taylor's series coefficients of functions, Toeplitz, Hermitian-Toeplitz and Hankel determinants come under the coefficient problems. For more details see~\cite{TplzSurya,Banga,Cudna,Kowa,Lecko1,Verma}.
    We also derived the sharp estimate of $\vert a_n\vert$ ($n=2,3,4$), certain Toeplitz, Hermitian-Toeplitz and Hankel determinants formed over the coefficient of $f \in \mathcal{S}^*_{nc}$.
    For $f(z) = z + \sum_{n=2}^\infty a_n z^n \in \mathcal{A}$, the $m^{th}$ Hankel and Hermitian-Toeplitz determinant for $m\geq 1$ and $n\geq 0$ are  respectively given by
\begin{equation}\label{hankel}
    H_{m}(n)(f) = \begin{vmatrix}
	a_n & a_{n+1} & \cdots & a_{n+m-1} \\
	{a}_{n+1} & a_{n+2} & \cdots & a_{n+m}\\
	\vdots & \vdots & \vdots & \vdots\\
    {a}_{n+m-1} & {a}_{n+m} & \cdots & a_{n+2m-2}
    \end{vmatrix},
\end{equation}
\begin{equation}\label{htplz}
     T_{m,n}(f)= \begin{vmatrix}
	a_n & a_{n+1} & \cdots & a_{n+m-1} \\
	\bar{{a}}_{n+1} & a_n & \cdots & a_{n+m-2}\\
	\vdots & \vdots & \vdots & \vdots\\
    \bar{{a}}_{n+m-1} & \bar{{a}}_{n+m-2} & \cdots & a_n
	\end{vmatrix},
\end{equation}
     where $\bar{a}_n = \overline{a_{n}}$.
      Toeplitz matrices have constant entries along their diagonals, while Hankel matrices have constant entries along their reverse diagonals.
      In particular, $$H_{2}(2)(f)= a_2 a_{4} -a_{3}^2, \;\; T_{3,1}(f) = 1 - 2 \lvert a_2 \rvert^2  +2 \RE (a_2^2 \bar{a}_3) - \lvert a_3 \rvert^2 $$
    $$ T_{2,1}(f) = 1 - \vert a_2 \vert^2 \; \text{and} \; H_3(1)(f) = a_3 (a_2 a_4 - a_3^2) - a_4 ( a_4 - a_2 a_3) + a_5 ( a_3 - a_2^2).$$
        Finding the sharp  bound of $\lvert H_{2}(2)(f)\rvert$ and $\lvert H_{3}(1)(f)\rvert$ for the class $\mathcal{S}$ and its subclasses has always been the focus of many researchers.
       Although, investigations concerning Hermitian Toeplitz are recently introduced in \cite{vasu,Cudna}. For more work in this direction, one can refer \cite{Lecko2,TplzSurya,tuneski} and the references cited therein. We obtain the sharp bounds of $\lvert H_{2}(2)(f)\rvert$, $\lvert H_{3}(1)(f)\rvert$, $T_{2,1}(f)$   and $T_{3,1}(f)$ when $f$ belongs to the class $\mathcal{S}^*_{nc}.$

    Further, we study the radius problems, inclusion relations with other subclasses of starlike functions and differential subordination implications for the same.

\section{Main results}
 Firstly, we prove that the class $\mathcal{S}^*_{nc}$ is non-empty using the following necessary and sufficient condition for a function $f$ to be in $\mathcal{S}^*_{nc}$.
\begin{theorem}
 A function $f\in\mathcal{S}_{nc}^*$ if and only if there exists an analytic function $q(z)\prec \varphi_{nc}
 (z)=(1+z)/\cos{z}$, such that
\begin{equation}\label{eq2}
    f(z)=z \exp{\int_0^z \frac{q(t)-1}{t} dt}.
\end{equation}
\end{theorem}
\begin{proof}
    Let $f\in\mathcal{S}_{nc}^*$ and let $zf'(z)/f(z):=q(z)\prec \varphi_{nc}
    (z)$. Then by integrating this equation we obtain (\ref{eq2}). If $f$ is given by (\ref{eq2}) with an analytic function $q$ such that $q(z) \prec \varphi_{nc}$, then by logarithmic differentiation of (\ref{eq2}), we obtain $zf'(z)/f(z)=q(z)$. Therefore $zf'(z)/f(z)\prec \varphi_{nc}(z)$ and $f\in\mathcal{S}_{nc}^*$.
\end{proof}
    Define the functions $f_n(z)$ such that $f_n(0)=f'_n(0)-1=0$ and
\begin{equation}\label{extremalf}
      \frac{z f'_n(z)}{f_n(z)}=\frac{1+z^{n-1}}{\cos{z^{n-1}}} , \;\; n=2,3,\cdots.
\end{equation}
     Then $f_n\in\mathcal{S}_{nc}^*$.

     In particular, if we take $q(z)=\varphi_{nc}(z)$, then the corresponding function
\begin{equation}\label{ext}
    \tilde{f}(z):=f_2(z)= z\exp{\int_0^z\frac{1+t-\cos{t}}{t\cos{t}}dt}=z+z^2+\frac{3}{4}z^3+\frac{7}{12}z^4+\frac{35}{96}z^5+\cdots
\end{equation}
      plays the role of an extremal function for many extremal problems for the class $\mathcal{S}^*_{nc}$. In~\cite{MaMi}, Ma and Minda proved some subordination results which yield that if $f\in\mathcal{S}^*_{nc}$ then $f(z)/z \prec \tilde{f}(z)/z$ and $zf'(z)/f(z)\prec z \tilde{f}'(z)/\tilde{f}(z)$ and therefore, the following result is obtained.
\begin{theorem}
 Let $f\in\mathcal{S}^{*}_{nc}$ and $\tilde{f}$ be the extremal function given by (\ref{ext}). Then the following holds:
\begin{enumerate}
    \item \textbf{Growth theorem:} For $\vert z_0\vert =r<1$, we have
    $$  - \tilde{f}(-r)\leq \vert f(z_0)\vert \leq  \tilde{f}(r).   $$
    Equality holds for some $z_0\neq 0$ if and only if $f$ is a rotation of $\tilde{f}$.
    \item \textbf{Rotation theorem:} For $\vert z_0\vert =r<1$, we have
    $$\vert \arg\{f(z_0)/z_0\}\vert \leq \max_{\vert z\vert =r} \arg\{\tilde{f}(z)/z\}.$$
    \item \textbf{Distortion theorem:} For $\vert z_0\vert =r<1$, we have
    $$\tilde{f}'(-r)\leq \vert f'(z_0)\vert \leq \tilde{f}'(r). $$
    Equality holds for some $z_0\neq 0$ if and only if $f$ is a rotation of $\tilde{f}$
\end{enumerate}
\end{theorem}
\begin{theorem}\label{thm1}
      For $\vert z\vert =r $,
       $\min \RE\{\varphi_{nc}
       (z)\}=\varphi(-r)$ \text{and} $\max \RE\{\varphi_{nc}(z)\}=\varphi(r). $
\end{theorem}
\begin{proof}
      The boundary point of $\varphi_{nc}
      (\mathbb{D}_{r_0})$, where $\mathbb{D}_{r_0}=\{z\in\mathbb{C}:\vert z\vert <r_0\}$, can be written in the form
      $$\varphi_{nc}(r_0 e^{i\theta})=\frac{1+r_0 e^{i\theta}}{\cos(r_0 e^{i\theta})}.$$
      The outward normal at the point $\varphi_{nc}(\zeta)$ is given by $\zeta \varphi_{nc}'(\zeta)$, where $\vert \zeta\vert =r_0$. Since we need to find the bounds of real part, it is sufficient to find the points at which, the imaginary parts of the normal, is a constant. Let us consider that $\delta(\theta):=\IM \{\zeta\varphi'(\zeta)\}$. A difficult computation shows that $\delta(\theta)=0$ for $\theta=0$ and $\theta=\pi$. It is easy to check that the maximum value is obtained at $\theta=0$ and minimum value is obtained at $\theta=\pi$. Since $r_0$ is arbitrary, for $r_0=r$ we get the desired result.
\end{proof}
\begin{theorem}\textbf{(Function's Bounds)}
    Let $\Phi_R$ and $\Phi_I$ denote the real and imaginary part of $\Phi_{nc}(z)= (1+z)/\cos{z}$ respectively. Then we have
\begin{enumerate}
  \item $ 0 \leq  \Phi_R \leq {(4 \cos{1})}/{(1 + \cos{2})} .$
  \item $ - \gamma_0 \leq  \Phi_I \leq  \gamma_0 $, where $\gamma_0\approx 1.6471.$
  \item $\vert \arg \varphi_{nc}(z) \vert \leq \pi/2.$
\end{enumerate}
\end{theorem}

\begin{figure}[h]
\centering
\captionsetup{justification=centering,margin=2cm}

   \includegraphics[width=4.5cm, height=4cm]{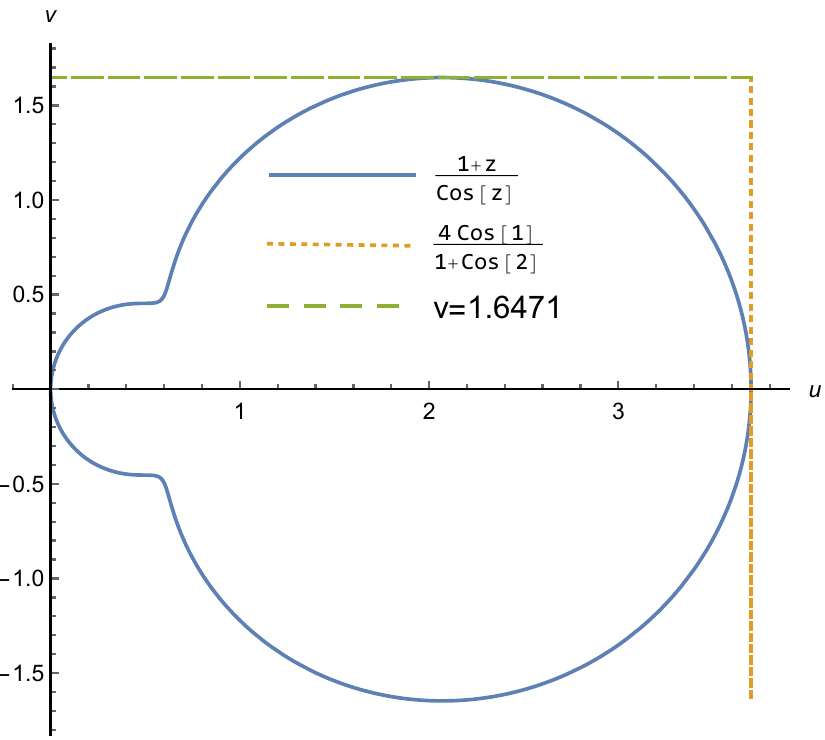}
 \caption{Bounds of $\varphi_{nc}$}
 \label{funbnd}
\end{figure}

\subsection{Hermitian-Toeplitz determinants}
    Giri and Kumar~\cite{TplzSurya} established the sharp bounds of $T_{2,1}(f)$ and $T_{3,1}(f)$ for $f$ belongs to the classes $\mathcal{S}^*(\varphi)$. The following results directly follow for the class $\mathcal{S}^*_{nc}.$
\begin{theorem}
    If $f\in \mathcal{S}^*_{nc}$, then
    $$ 0 \leq \det T_{2,1}(f) \leq 1 \;\; \text{and} \;\; \det T_{3,1}(f) \leq 1.$$
    All the estimates are sharp.
\end{theorem}
\begin{theorem}
    If $f\in \mathcal{S}^*_{nc}$, then $  \det T_{3,1}(f) \geq -1/15 $. The bound is sharp.
\end{theorem}

\subsection{Hankel Determinants}
   Let $\mathcal{P}$ denote the Carath\'{e}odory class containing analytic functions $p(z)= 1+ \sum_{n=1}^\infty p_n z^n$ in $\mathbb{D}$ such that $\RE p(z) >0 $ for all $z\in \mathbb{D}$. The coefficients estimates of Carath\'{e}odory functions prominently used to deal with the coefficient problems.
    The following lemmas for the class $\mathcal{P}$ are used to establish the results.
\begin{lemma}\cite{MaMi}\label{firstL}
   If $p(z)=1+\sum_{n=2}^\infty p_n z^n\in\mathcal{P}$, then for a complex number $\nu$
  $$\vert p_2-\nu p_1^2\vert \leq 2 \max{(1,\left\vert 2\nu-1\right\vert )}. $$
  This result is sharp for the function
  $p(z)={(1+z^2)}/{(1-z^2)}$ \text{and} $p(z)= {(1+z)}/{(1-z)}. $
\end{lemma}
\begin{lemma}\cite{RaVa}\label{secondL}
$[5]$ Let $p(z)=1+\sum_{n=2}^\infty p_n z^n\in\mathcal{P}$, then for $n,m\in \mathbb{N}$
\[
\vert p_{n+m}-\nu p_n p_m\vert \leq  \left\{
\begin{array}{ll}
      2, & 0\leq \nu\leq 1 \\
      2\vert 2\nu-1\vert , & \text{elsewhere}.\\

\end{array}
\right.
\]
\end{lemma}
    The next lemma gives the representations of $p_2$, $p_3$ and $p_4$ in terms of $p_1$. The expressions for $p_2$ and $p_3$ were given in~\cite{LiZl1,LiZl2} and for $p_4$ in~\cite{KwSiLe}.
\begin{lemma}\label{lemma3}
   Let $p(z)=1+\sum_{n=2}^\infty p_n z^n\in\mathcal{P}$, then
\begin{align*}
   2 p_2 &= p_1^2 + \gamma (4-p_1^2), \\
   4 p_3 &= p_1^3 + 2(4-p_1^2)p_1 \gamma - (4- p_1^2) p_1 \gamma^2 + 2 (4-p_1^2) (1-\vert  \gamma \vert ^2) \eta \;\; \text{and} \\
   8 p_4 &= p_1^4 + (4 - p_1^2) \gamma ( p_1^2 (\gamma^2 - 3 \gamma + 3) + 4 \gamma) - 4 (4 - p_1^2) (1 - \vert \gamma \vert^2 ) ( p_1 (\gamma - 1) \eta \\
         &+ \bar{\gamma} \eta^2 - (1 - \vert \gamma \vert^2) \rho)
\end{align*}
   for some complex numbers $\gamma, \eta, \rho$ with $\vert \gamma \vert \leq 1$, $\vert  \eta \vert \leq 1$ and $\vert \rho \vert \leq 1$.
\end{lemma}

\begin{theorem}
  If $f(z)=z+\sum_{n=2}^\infty a_n z^n\in\mathcal{S}^*_{nc}$, then
\begin{equation*}
    \sum_{n=2}^\infty (n^2 k_1-4)\vert a_n\vert ^2\leq (4-k_1),
\end{equation*}
where $k_1=\cos^2{1}$.
\end{theorem}
\begin{proof}

 Since $f\in\mathcal{S}^*_{nc}$, therefore by the definition of subordination
    $$\frac{zf'(z)}{f(z)}=\frac{1+\omega(z)}{\cos{(\omega(z))}}, $$
    where $\omega(z)$ is a Schwarz function. Using the identity
    $$f^2(z)=(zf'(z))^2\frac{\cos^2{w(z)}}{(1+w(z))^2} ,$$
    we obtain
    $$2\pi\sum_{n=1}^\infty \vert a_n\vert ^2 r^{2n}=\int_{0}^{2\pi} \vert f(re^{i\theta})\vert ^2 d\theta=\int_0^{2\pi}\vert re^{i\theta} f'(re^{i\theta})\vert ^2\bigg\vert \frac{\cos{w}}{1+w}\bigg\vert ^2. $$
    After some simple computation, we get
    $$ 2\pi \sum_{n=1}^\infty \vert a_n\vert ^2 r^{2n}\geq \frac{\cos^2{1}}{4}\int_0^{2\pi}\vert re^{i\theta}f'(re^{i\theta})\vert ^2 d\theta=2\pi \frac{\cos^2{1}}{4}\sum_{n=1}^\infty n^2 \vert a_n\vert ^2 r^{2n}, $$
    where $0<r<1$ and $a_1=1$. Thus
    $$\sum_{n=2}^\infty \bigg(n^2\frac{\cos^2{1}}{4}-1\bigg)\vert a_n\vert ^2 r^{2n}\leq 0. $$
    On letting $r\rightarrow 1^-$, we get the required result.
\end{proof}

\begin{corollary}
     Let $f(z)=z+\sum_{n=2}^\infty a_n z^n\in\mathcal{S}^*_{nc}$ and $k_1=\cos^2{1}$ , then
     $$\vert a_n\vert \leq \sqrt{\frac{4-k_1}{n^2\cos^2{1}-4}},\quad n=2,3,4,\cdots. $$
\end{corollary}
   The following sharp bound of Fekete-Szeg\"{o} functional for $f\in \mathcal{S}^*_{nc}$ can be deduced directly from~\cite{AliFS}.
\begin{theorem}\label{thmFS}
Let $f\in \mathcal{S}_{nc}^*$, then
\[
\vert a_3-\mu a_2^2\vert \leq  \left\{
\begin{array}{ll}
     -\mu+\frac{3}{4}, & \mu<\frac{1}{4} \\
      \frac{1}{2}, & \frac{1}{4}\leq \mu\leq \frac{5}{4}\\
      \mu-\frac{3}{4} & \mu>\frac{5}{4}. \\
\end{array}
\right.
\]
  The bound is sharp.
\end{theorem}

\begin{theorem}
If the function $f(z)=z+\sum_{n=2}^\infty a_nz^n$ belongs to the class $\mathcal{S}^*_{nc}$, then
   $$ \vert a_2\vert \leq 1,  \quad    \vert a_3\vert \leq\frac{3}{4}, \quad  \vert a_4\vert \leq\frac{7}{12}, \quad \vert a_5\vert \leq \frac{1}{3} .  $$
  The bounds of $\vert a_n \vert$ for $n=2,3,4$ are sharp.
\end{theorem}
\begin{proof}
   If $f \in\mathcal{S}^*_{nc}$, then there exists a Schwarz function $\omega(z)$ satisfying $\omega(0)=0$ and $\vert \omega(z) \vert \leq \vert z \vert$ such that
\begin{equation}\label{eq11}
    zf'(z)=f(z)\bigg(\frac{1+\omega(z)}{\cos{\omega(z)}}\bigg) , \quad z \in \mathbb{D}.
\end{equation}
   By the one-to-one correspondence between the classes of Schwarz functions and the class $\mathcal{P}$, let there is $p(z) = 1 + \sum_{n=1}^\infty p_n z^n \in \mathcal{P}$ such that $\omega(z) = (p(z)-1)/(p(z)+1).$
    Then by substituting the $f(z)$ and $p(z)$ in (\ref{eq11}) and comparing the coefficients, we obtain
\begin{equation}\label{a5}
\begin{aligned}
    a_2 = \frac{p_1}{2}, \; a_3 &= \frac{p_1^2 + 4 p_2}{16}, \; a_4 = \frac{p_1^3 + 4 p_1 p_2 + 16 p_3}{96}\\
    \text{and} \;\; a_5 &= \frac{ -p_1^4 + 4 p_1^2 p_2 + 4 p_1 p_3 + 24 p_4}{192}.\\
\end{aligned}
\end{equation}
    It is well known that the coefficients of Carath\'{e}odory function $p(z)$ satisfy the inequality $\vert p_n\vert \leq 2$, $(n=1,2,3,\cdots)$, which gives $\vert a_2\vert \leq 1$. The bound of $\vert a_3\vert$ directly follow from Theorem~\ref{thmFS} for $\mu =0$. Further, the bound of $\vert a_4 \vert$ can be obtained as a special case of Ali et al.~\cite[Theorem 1]{AliFS}.
    Now, the fifth coefficient
\begin{align*}
    a_5&=\frac{1}{192}( -p_1^4 + 4 p_1^2 p_2 + 4 p_1 p_3 + 24 p_4).\\
    &=\frac{1}{192}[4 p_1^2(p_2-\frac{1}{4}p_1^2)+4(p_1p_3+6p_4)].
\end{align*}
    Applying the triangular inequality together with Lemma \ref{firstL} and Lemma \ref{secondL}, we obtain
    $$\vert a_5\vert \leq \frac{1}{192}(8\vert p_1\vert ^2+ 32). $$
    Using the fact that $\vert p_1\vert \leq 2$, we get the required bound of $\vert a_5 \vert.$

    The function $\tilde{f}$ given by (\ref{ext}) plays the role of an extremal function for $\vert a_n\vert$ when $n=2,3$ and $4$.
\end{proof}
\begin{theorem}
    If $f\in\mathcal{S}^*_{nc}$, then
     $$\vert H_{2}(2)(f)   \vert \leq \frac{1}{4}.$$
    The result is sharp.
\end{theorem}
\begin{proof}
     For $f\in \mathcal{S}^*_{nc}$, from (\ref{a5}), we can get
     $$a_2 a_4-a_3^2=\frac{1}{768} (p_1^4 - 8 p_1^2 p_2 - 48 p_2^2 + 64 p_1 p_3). $$
    Since the class $\mathcal{P}$ is invariant under rotation, we can take $p_1\in [0,2]$. Let $p:=p_1$ and substituting the value of $p_2$ and $p_3$ from Lemma \ref{lemma3}, we have
\begin{align*}
 \vert a_2 a_4-a_3^2\vert    =  &\frac{1}{768} \vert (p^4 - 4 p^2 (p^2 + (4 - p^2) x) -
   12 (p^2 + (4 - p^2) x)^2\\
   &+16 p (p^3 + 2 p (4 - p^2) x - p (4 - p^2) x^2 +
      2 (4 - p^2) (1 - \vert x\vert ^2) z))\vert .\\
      =&\frac{1}{768}\vert (p^4+12 p^2(4-p^2)x+(28p^4-96p^2+192)x^2\\
      &+32 p(4-p^2)(1-\vert x\vert ^2)z)\vert .
\end{align*}
   Now, applying the triangle inequality and replacing $\vert x\vert $ by $\varrho$, we obtain
\begin{align*}
 \vert a_2 a_4-a_3^2\vert    \leq {} &\frac{1}{768}(p^4+12 p^2(4-p^2)\varrho +(28p^4+32 p^3-96p^2\\
 &-128 p+192)\varrho^2+32 p(4-p^2))\\
 &=: G(p,\varrho) \quad \text{(say)}.
\end{align*}
    On differentiation $G(p,\varrho)$ with respect to $\varrho$, we obtain
          $$\frac{\partial G(p,\varrho)}{\partial \varrho}=\frac{1}{768} (12 p^2 (4 - p^2) + 2 (192 - 128 p - 96 p^2 + 32 p^3 + 28 p^4) \varrho).$$
    Since ${\partial G(p,\varrho)}/{\partial \varrho}>0$, therefore $G(p,\varrho)$ is an increasing function and attains its maximum value at $\varrho=1.$ Hence
           $$G(p,1)=:G(p)=\frac{1}{768}(192 - 48 p^2 + 17 p^4).$$
   A simple calculation yields that the maximum value of $G(p)$ occurs at $p=0$. Thus,
    $$ \vert H_{2}(2)(f) \vert = \vert a_2 a_4-a_3^2\vert \leq \frac{192}{768}=\frac{1}{4}.$$
    The bound is sharp for the function $f_3(z)$ given by (\ref{extremalf}).
\end{proof}
\begin{theorem}
    If $f\in \mathcal{S}^*_{nc}$, then
    $$  \vert H_{3}(1)(f) \vert \leq \frac{1}{9}.$$
    The bound is sharp.
\end{theorem}
\begin{proof}
    Let $f(z)= z+ \sum_{n=2}^\infty a_n z^n \in \mathcal{S}^*_{nc}$. Then from (\ref{hankel}), we have
    $$  H_3(1)(f) = a_3 (a_2 a_4 - a_3^2) - a_4 ( a_4 - a_2 a_3) + a_5 ( a_3 - a_2^2).$$
    Using the values of $a_2$,  $a_3$, $a_4$ and $a_5$ from (\ref{a5}) together with Lemma \ref{lemma3} and $p_1=p \in[0,2]$, we obtain
    $$ H_3(1)(f) = \frac{1}{36864} \bigg(\vartheta_1( p, \gamma) + \vartheta_2(p, \gamma) \eta + \vartheta_3( p, \gamma) \eta^2 + \psi(p, \gamma, \eta) \rho \bigg) , $$
    where
\begin{align*}
   \vartheta_1(p,\gamma ) &=  5 p^6 - 26 \gamma  (4 - p^2 ) p^4 - 144 \gamma^2 (4-p^2 ) p^2 + 56 \gamma^2 (4 - p^2 ) p^4 + 68 \gamma^2 (4 - p^2 )^2 p^2 \\
                    &-36 \gamma^3 (4 - p^2 ) p^4   -40 \gamma ^3  (4 - p^2 )^2 p^2 + 8 \gamma ^4  (4-p^2 )^2 p^2  ,\\
   \vartheta_2(p,\gamma ) &= (4 - p^2) (1 - \vert\gamma\vert^2) (-40 p^3 + 144 p^3 \gamma + 80 p (4 - p^2) \gamma - 32 p (4 - p^2) \gamma^2), \\
   \vartheta_3(p,\gamma ) &= (4 - p^2) (1 - \vert\gamma\vert^2) (-256 (4 - p^2) - 32 (4 - p^2) \vert\gamma\vert^2 + 144 p^2 \bar{\gamma}), \\
   \psi(p,\gamma, \eta ) &= (4 - p^2) (-144 p^2 + 288 (4 - p^2) \gamma) (1 - \vert \gamma\vert^2 ) (1 - \vert\eta\vert^2),
\end{align*}
   $\vert \gamma \vert \leq 1$, $\vert  \eta \vert \leq 1$ and $\vert \rho \vert \leq 1$ and $p=p_1\in [0,2]$ by the rotationally invariant property of the class $\mathcal{P}$.
     Now, by taking $\vert \gamma \vert :=x $, $\vert \eta \vert :=y $ and keeping in mind $\vert \rho \vert \leq 1$, we get
     $$ \vert H_3(1)(f)\vert \leq \frac{1}{36864} \bigg(\vert\vartheta_1( p, \gamma) \vert +\vert \vartheta_2(p, \gamma)\vert y + \vert \vartheta_3( p, \gamma)\vert y^2 +  \vert \psi(p, \gamma, \eta) \vert \bigg)=: G(p,x,y) , $$
     where
     $$ G(p,x,y) = \frac{1}{36864} \bigg(g_1(p, x) + g_2(p, x) y + g_3(p, x) y^2 + g_4(p, x) \bigg) $$
      such that
\begin{align*}
     g_1(p,x) &= 5 p^6 + 26 p^4 (4 - p^2) x + 144 p^2 (4 - p^2) x^2 +  56 p^4 (4 - p^2) x^2 + 68 p^2 (4 - p^2)^2 x^2 \\
              &+ 36 p^4 (4 - p^2) x^3 + 40 p^2 (4 - p^2)^2 x^3 + 8 p^2 (4 - p^2)^2 x^4 , \\
     g_2(p,x) &= (4 - p^2) (1 - x^2) (40 p^3 + 144 p^3 x + 80 p (4 - p^2) x + 32 p (4 - p^2) x^2) , \\
     g_3(p,x) &= (4 - p^2) (1 - x^2) (256 (4 - p^2) + 32 (4 - p^2) x^2 + 144 p^2 x) , \\
     g_4(p,x) & = (4 - p^2) (1 - y^2) (144 p^2 + 288 (4 - p^2) x) (1 - x^2) .
\end{align*}
    Thus, we obtain
\begin{align*}
    G(p,x,y) &= \frac{1}{36864} \bigg(5 p^6 + 26 p^4 (4 - p^2) x + 144 p^2 (4 - p^2) x^2 + 56 p^4 (4 - p^2) x^2 \\
            & + 36 p^4 (4 - p^2) x^3  + 68 p^2 (4 - p^2)^2 x^2 + 40 p^2 (4 - p^2)^2 x^3  + 8 p^2 (4 - p^2)^2 x^4  \\
           & + (4 - p^2) (1 - x^2) (40 p^3 + 144 p^3 x + 80 p (4 - p^2) x + 32 p (4 - p^2) x^2) y  \\
           & + (4 - p^2) (1 - x^2) (256 (4 - p^2) + 144 p^2 x + 32 (4 - p^2) x^2) y^2  \\
           &+ (4 - p^2) (144 p^2 + 288 (4 - p^2) x) (1 - x^2) (1 - y^2) \bigg).
\end{align*}
    Clearly, the maximum value of $G(p,x,y)$ in the closed cuboid $\Gamma: [0,2]\times [0,1]\times [0,1]$ is the required bound of $\vert H_3(1)(f)\vert$. We establish this by finding the maximum value of $G(p,x,y)$ in the interior of six faces, on the twelve edges and in the interior of $\Gamma.$

   \textbf{Case I.} Firstly, we proceed with interior points of $\Gamma$. Let $(p, x, y) \in (0, 2) \times (0, 1) \times (0, 1)$. A simple computation yields
\begin{align*}
     \frac{\partial G}{\partial y} &= \frac{(4 - p^2)}{2304} \bigg(  9 p^2 ( x (1 - x^2) (1 + 2 y) -2 (4 - p^2) y ) + 2 (4 - p^2) (-18 (4 - p^2) x y  \\
   &+ 8 (1 - x^2 + 2 (1 - x^2) y) +  x^2 (1 - x^2 + 2 (1 - x^2) y) )  \bigg).
\end{align*}
   Now $\partial G/\partial y= 0$ gives
   $$ y_0 = \frac{ p^3 (-4 x^2 + 8 x + 5 )+8 p x (2 x + 5)}{4 (x - 1) (p^2 (2 x - 25) -8 (x-8) )}. $$
    The critical point $y_0$ lie in the interval $(0,1)$ whenever
\begin{equation}\label{ineq1}
     p^3 ( 4 x^2 - 8 x - 5  ) + 4 p^2  (2 x^2 - 27 x + 25 ) - 8 p x (2 x + 5) - 32  (x^2 - 9 x + 8 )>0
\end{equation}
    and
\begin{equation}\label{ineq2}
    p^2 (25 - 2x) > 8 (8-x)
\end{equation}
    hold.

    It can be simply observe that as $p\rightarrow 2$, there exists at least one $x\in (0,1)$ such that the inequality (\ref{ineq1}) is not true. Similarly, when $x \rightarrow 0$, there exists at least one $p\in (0,2)$ such that inequality (\ref{ineq1}) does not remain valid.
    The inequality (\ref{ineq1}) is true only for $p=2$ and $x\in  (0,13/36)$, which is outside the considered domain. Hence, we deduce that $G(p,x,y)$ has no critical point.

\textbf{Case II.} Secondly, we consider the interior of all the six faces of the cuboid $\Gamma$.
     On the face $p=0$, let $G(0,x,y) = h_1(x,y)$, where
\begin{align*}
   h_1(x,y) = \frac{1}{72} ( 1 - x^2 ) (x^2 y^2 - 9 x (y^2 - 1 )+8 y^2 )
\end{align*}
  for $x, y \in (0,1)$. A simple calculation shows that
  $$ \frac{\partial h_1}{\partial y} =  \frac{1}{72} ( 1 - x^2 ) (2 x^2 y - 18 x y + 16 y ) \neq 0 \quad  \forall\;\; x,y \in (0,1).  $$
  Hence $h_1(x, y)$ has no critical point in $(0,1 ) \times (0, 1).$

   On the face $p=2$, we simply get
\begin{equation}\label{p=2}
    G(2,x,y)=  \frac{5}{576}.
\end{equation}
    On the face $x=0$,  let $G(p,0,y)=h_2(p,y)$, where
\begin{equation}\label{h2(p,y)}
   h_2(p, y) = \frac{5 p^6 + 144 (4 - p^2 ) p^2 (1 - y^2 ) + 256 (4 - p^2 )^2 y^2 + 40 (4 - p^2 ) p^3 y}{36864},
\end{equation}
   $p \in (0,2) $ and $ x \in (0, 1).$
   The critical points of $h_2(p,y)$ are given by the solution of $\partial h_2(p,y)/\partial p=0$ and $\partial h_2(p,y)/\partial y=0$. The equation $\partial h_2(p,y)/\partial y=0$ satisfies for
   $$ y_0 = \frac{5 p^3}{4 (25 p^2-64 )}. $$
   The root $y_0$ lies in the range $(0,1)$ whenever $p \in (1.67136, 2)$. Also, $\partial h_2(p,y)/\partial p =0$ implies
   $$  \frac{p (625 p^8 - 18050 p^6 + 110320 p^4 - 251904 p^2 + 196608 )}{1536 (64 - 25 p^2 )^2} =0,$$
   which has no solution for the range $p\in (1.67136, 2)$. Consequently, $h_2(p,y)$ has no critical point for $p\in (0,2)$ and $x\in (0,1).$

    On the face $x=1$, let $G(p,1,y)= h_3(p,y)$ where
\begin{equation}\label{h3(p,y)}
    h_3(p,y) = \frac{5 p^6 + 144 p^2 (4 - p^2) + 118 p^4 (4 - p^2) + 116 p^2 (4 - p^2)^2}{36864}, \;\; p\in (0,2).
\end{equation}
   A straightforward computation reveals that the maximum of $h_3(p,y)$ is attained at  $p = 2 \sqrt{2(25-\sqrt{587} )/3}$ given by
\begin{equation}\label{maxh3}
   \max h_3(p,y)= \frac{587 \sqrt{587}-14200 }{324}.
\end{equation}

    On the face $y=0$, let  $G(p,x,0)= h_4(p, x)$, where
\begin{align*}
   h_4(p,x) & = \frac{1}{36864} \bigg(5 p^6 + 26 p^4 (4 - p^2) x + 144 p^2 (4 - p^2) x^2 + 56 p^4 (4 - p^2) x^2 \\
           &+ 68 p^2 (4 - p^2)^2 x^2 +   36 p^4 (4 - p^2) x^3 + 40 p^2 (4 - p^2)^2 x^3 \\
           & + 8 p^2 (4 - p^2)^2 x^4 + (4 - p^2) (144 p^2 +
      288 (4 - p^2) x) (1 - x^2) \bigg) .
\end{align*}
   for $x\in( 0,1)$ and $p \in (0,2 )$. A simple calculation yields that $\partial h_4/\partial p =0$ and $\partial h_4/\partial x =0$ have no solutions for $p \in (0,2)$ and $x \in (0,1)$. Hence, $h_4(p,x)$ does not have any critical point.

    On the face $y=1$, let $G(p,x,1)= h_5(p,x)$ such that
\begin{align*}
   h_5(p,x) &=  \frac{1}{36864} \bigg(5 p^6 + 26 p^4 (4 - p^2) x + 144 p^2 (4 - p^2) x^2 + 56 p^4 (4 - p^2) x^2 \\
          & + 68 p^2 (4 - p^2)^2 x^2 + 36 p^4 (4 - p^2) x^3 + 40 p^2 (4 - p^2)^2 x^3 + 8 p^2 (4 - p^2)^2 x^4 \\
          & + (4 - p^2 ) (1 - x^2 ) (32 (4 - p^2 ) x^2 + 144 p^2 x + 256 (4 - p^2 ) )+ (4 -  p^2 ) (1 - x^2 ) (144 p^3 x \\
           & + 40 p^3 + 32 (4 - p^2 ) p x^2 + 80 (4 - p^2 ) p x )  \bigg).
\end{align*}
   Proceeding as in the above case, we obtain that $h_5(p,x)$ has no critical point for $p \in (0,2)$ and $x \in (0,1)$.

\textbf{Case III.} Now, we find the maximum value of $G(p,x,y)$ on the eight edges of the cuboid $\Gamma.$ In view of the function $G(p,x,y)$, we have
   $$ G(p, 0, 0) = {5 p^6 + 144 (4 - p^2 ) p^2}/{36864} =: k_1(p). $$
   Clearly, the maximum value of $k_1(p)$ is attained at $p=2 \sqrt{ 2(6-\sqrt{21} )/5}$ and given by
   $$ \max k_1(p)= \frac{1}{300} (7 \sqrt{21}-27 )= 0.0169268     . $$
   Again considering the function $G(p,x,y)$, we obtain $G(p,0,1)  = {(5 p^6 + 256 (4 - p^2 )^2 + 40  (4 -  p^2 ) p^3)}/{36864} =: k_2(p)$ for $p \in (0,2)$. Since $k_2(p)$ is a decreasing function, therefore maximum will attain at $p=0$. Thus
   $$ \max k_2(p) = \frac{1}{9}  .$$
   From (\ref{h3(p,y)}), it can be observed that $G(p,1,y)$ does not depend on $y$, hence we directly obtain $G(p,1,0)= G(p,1,1) = h_3(p,y)=: k_3(p). $  Thus from (\ref{maxh3}), we get that $\max k_3(p)=  \max h_3(p,y)= ({587 \sqrt{587}-14200 })/{324}. $

   On substituting $p=0$ in (\ref{h3(p,y)}), we get $G(0,1,y)= 0 .$ Also for $p=2$, from (\ref{p=2}) note that  $G(2, x, y)$ is independent from of all the variables $x$ and $y$. Therefore the value of $G(p,x,y)$ on the edges $p=2$, $x=0$; $p=2$, $x=1$; $p=2$, $y=0$ and $p=2$, $y=1$ given by
   $$ G(2,0,y) =  G(2,1 ,y) =G(2,x, 0)= G(2,x, 1)=  \frac{5}{576}  $$
   respectively. From (\ref{h2(p,y)}) at $p=0$,  we obtain $G(0,0, y) = y^2/9 =: k_4 (y).  $  It is easy to verify that
   $$ \max k_4(y) = \frac{1}{9}.  $$
   Again, for $p=0$ and $y=1$, we have $G(0, x, 1) =  (1 - x^2 ) (x^2 + 8 )/72=: k_5(x)$. Since $k_5(x)$ is a decreasing function of $x$, hence
   $$ \max k_5(x) = \frac{1}{9}.   $$
   Substituting $p=y=0$ in $G(p,x,y)$, we get $G(0,x,0) = x (1 - x^2)/8=:k_6(x)$. It is a simple exercise to check that the maximum value of $k_6(x)$ is attained at $x = 1/\sqrt{3}$ given by
   $$ \max k_6(x) =  \frac{1}{12 \sqrt{3}}.  $$

   In view of the above all three cases, the maximum value $1/9$ of $G(p,x,y)$ is attained at the edge $p=x=0$ and $y=1$, which is the required bound. The sharpness of the bound follows from the function $f_4(z)$ given by (\ref{extremalf}).
\end{proof}

\subsection{Inclusion relations and Radius problems}

    In 1999, Kanas and Wi\'{s}niowska~\cite{KanWis} introduced the class
    $$ k-  \mathcal{ST} = \bigg\{ f \in \mathcal{A}: \RE \frac{z f'(z)}{f(z)}  > k  \bigg\vert \frac{z f'(z)}{f(z)} - 1 \bigg\vert  \bigg\} $$
     of $k-$starlike functions. Geometrically, the boundary of the domain $\Omega_{k} = \{ w \in \mathbb{C}: \RE w > k \vert w -1 \vert $ represents an ellipse for $k >1$, a parabola for $ k= 1$ and a hyperbola for $k >1$. The next result shows the inclusion relation with the class $k-  \mathcal{ST}$ and the following classes (see \cite{Janow,parabola,Urela})
     $$  \mathcal{S}^*(M) =  \bigg\{ f\in \mathcal{A}:  \bigg\vert \frac{zf'(z)}{f(z)} - M \bigg\vert <M  \bigg\}, \quad M > \frac{1}{2},$$
   $$  \mathcal{ST}_p(a) = \bigg\{ f \in \mathcal{A}: \RE \bigg( \frac{z f'(z)}{f(z)} + a \bigg) > \bigg\vert \frac{z f'(z)}{f(z)} - a \bigg\vert,\; a > 0  \bigg\} $$
     and
    $$   \mu(\beta)= \bigg\{ f\in \mathcal{A}: \RE \bigg( \frac{z f'(z)}{f(z)} \bigg) < \beta, \; \beta >1 \bigg\}. $$
  A function $f$ in the class $\mathcal{S}^*(M)$ is called $M-$starlike function.

\begin{theorem}
    $k- \mathcal{ST} \subset  \mathcal{S}^*_{nc}$ for $k \geq {(4 \cos{1})}/{(4 \cos{1} - \cos{2} -1)}.$
\end{theorem}
\begin{proof}
     Let $f \in K-\mathcal{ST}$ and $ \Omega_k = \{ w \in \mathbb{C}: \RE w > k \vert w - 1 \vert \} $. For $k >1$, the boundary curve $\partial \Omega_k $ is an ellipse $\gamma_k : x^2 = k^2 ( x - 1)^2 + k^2 y^2$, which can be rewritten as
     $$  \frac{(x -x_0)^2}{u^2} +  \frac{(y -y_0)^2}{v^2} =1,$$
     where $x_0 = k^2/(k^2-1)$, $y_0 =0$, $u= k/(k^2 -1)$ and $v = 1/\sqrt{k^2 -1}$. Since $u > v$, therefore for the ellipse $\gamma_k$  to lie inside $\varphi_{nc}$, the range of $x_0 + u$ must be less than or equal to
     $ {( 4 \cos{1}) }/{(1 + \cos{2})}. $ A simple calculation shows that this condition holds whenever
     $ k \geq {(4 \cos{1})}/{(4 \cos{1} - \cos{2} -1)} .  $
\end{proof}
\begin{theorem}
    $\mathcal{S}^*_{nc} \subset \mathcal{ST}_{p}(a)$ whenever $a \geq  a_0 \approx  0.402301.$
\end{theorem}
\begin{proof}
   As we know that the boundary $\partial \Omega_a$ of the domain $\Omega_a = \{ w \in \mathbb{C}: \RE w + a > \vert w - a \vert \}$ denotes a parabola. So, $\mathcal{S}^*_{nc} \subset \mathcal{ST}_{p}(a)$, provided $\RE w + a > \vert w - a \vert $, where $w = (1 +z)/\cos{z}$. Taking $z = e^{i \theta}$, we get $T(\theta) < a$, where
\begin{align*}
    T(\theta) = \bigg(&(\cos (\theta)+1) \sinh (\sin (\theta)) \sin (\cos (\theta))+\sin (\theta) \cos (\cos (\theta)) \cosh (\sin (\theta)) \bigg)^2 \bigg/\bigg( 2 \\
                &(\cos (2 \cos{\theta})+\cosh (2 \sin{\theta}) ) ((\cos{\theta}+1) \cos (\cos {\theta}) \cosh (\sin{\theta})\\
                &-\sin{\theta} \sinh (\sin{\theta}) \sin (\cos{\theta})) \bigg).
\end{align*}
   Clearly, the graph of $T(\theta)$ in Figure \ref{figTT} shows that the maximum value of $T(\theta)$ is $T(\theta_0) \approx 0.402301$, where $\theta_0\approx 0.665124$ is the root of $T'(\theta) =0$. Since $\mathcal{ST}_{p}(a_1) \subset \mathcal{ST}_{p}(a_2)$ for $a_1 < a_2$, therefore $\mathcal{S}^*_{nc} \subset \mathcal{ST}_{p}(a)$ for $a \geq b \approx  0.402301.$
\begin{figure}[h]
\centering

   \includegraphics[width=6cm, height=4cm]{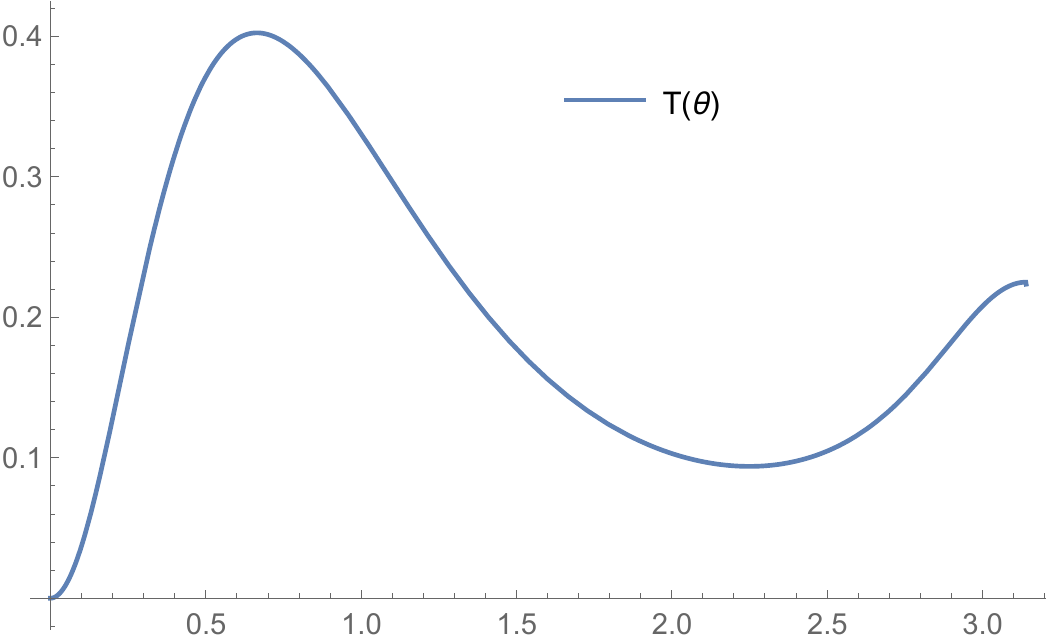}
 \caption{Graph of $T(\theta)$}
 \label{figTT}
\end{figure}
\end{proof}
   The following result can be directly deduced  from figure \ref{funbnd}.
\begin{theorem}
    The class $\mathcal{S}^*_{nc}$ satisfies the following relations:
\begin{enumerate}
  \item $\mathcal{S}^*_{nc} \subset \mathcal{S}^*$.
  \item $\mathcal{S}^*_{nc} \not\subset \mathcal{S}^*(\alpha)$ for all $\alpha \in (0,1)$.
  \item $ \mathcal{S}^*_{nc} \subset \mu(\beta)$ whenever $\beta \geq 2 \sec{1}.$
\end{enumerate}
\end{theorem}
\begin{figure}[h]\label{figinc}
\begin{tabular}{c}
\includegraphics[scale=0.34]{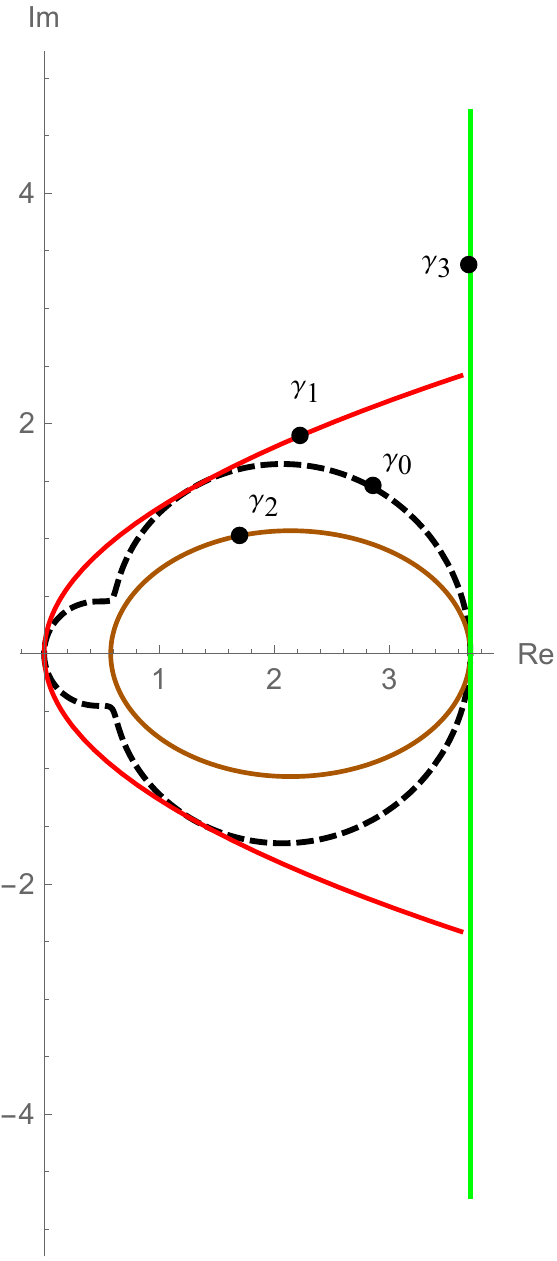}
\end{tabular}
\begin{tabular}{l}
\parbox{0.25\linewidth}{
{\bf \underline{Legend} - }\\  \vspace{0.15cm}
$\gamma_{0}: \frac{1+z }{\cos{z}} $}\\ \vspace{0.15cm}
$\gamma_{1}:  \vert \omega - 0.4023 \vert - \RE{(\omega)} = 0.4023 $ \\ \vspace{0.15cm}
$\gamma_{2}: \RE{\omega} = (4 \cos{1})/(4 \cos{1} - \cos{2} - 1) \vert \omega - 1 \vert $ \\ \vspace{0.15cm}
$\gamma_{3}:  \RE{\omega} = 2 \sec{1} $\\ \vspace{0.15cm}
\end{tabular}
\caption{Inclusion graphs}
\label{f2}
\end{figure}
\begin{theorem}
   Let $f \in \mathcal{S}^*_{nc}$, then the following hold:
\begin{enumerate}
  \item $f\in \mathcal{S}^*(\alpha)$ in $\vert z \vert \leq r_\alpha $, where $r_\alpha$ is the root of $(1-r)- \alpha \cos{r}=0.$
  \item $f\in \mu (\beta)$ in $\vert z \vert < r_\beta$, where
  \begin{equation}\label{rbeta}
  \begin{aligned}
   r_\beta = &
\left\{
\begin{array}{ll}
    r_0(\beta), &  1 < \beta < 2/(\cos{1}), \\
    1, &  2/(\cos{1}) \leq \beta
\end{array}
\right.
\end{aligned}
\end{equation}
   and $r_0(\beta)$ is the smallest root of
  $$ 1 + r = \beta \cos{r}. $$
\end{enumerate}
   The radii are sharp.
\end{theorem}
\begin{proof}
\begin{enumerate}
  \item    Let $f\in \mathcal{S}^*_{nc}$, then
    $$ \frac{z f'(z)}{f(z)} \prec \frac{1+z}{\cos{z}}.$$
    Therefore for $\vert z \vert =r <1$, by Theorem \ref{thm1}, we have
    $$ \frac{1-r}{\cos{r}} \leq \RE \frac{z f'(z)}{f(z)} \leq   \frac{1+r}{\cos{r}}. $$
    Thus, $f\in \mathcal{S}^*(\alpha)$, whenever
    $$ \RE \frac{z f'(z)}{f(z)}  \geq \frac{1-r}{\cos{r}}  > \alpha. $$
    This inequality is true for $r \in (0, r_\alpha$), where $r_\alpha$ is the root of $(1-r)- \alpha \cos{r}=0.$ To see the sharpness consider the function $\tilde{f}(z)$ given by (\ref{ext}) as $\RE (z \tilde{f}'(z)/\tilde{f}(z)) = \alpha$ for $z= - r_\alpha.$
  \item  Similarly, $f \in \mu(\beta)$, whenever
  $$ \RE \frac{z f'(z)}{f(z)} \leq   \frac{1+r}{\cos{r}} < \beta, \quad (\beta >1). $$
  The above inequality holds for $r\in (0, r_\beta)$, where $r_\beta$ is given by (\ref{rbeta}). Sharpness of the radius $r_\beta$ follows from the function $\tilde{f}$ defined in (\ref{ext}).
\end{enumerate}
\end{proof}
\begin{theorem}
   If $f\in \mathcal{S}^*_{nc}$, then $f$ is  convex function of order $\alpha$ for $\vert z \vert < r_c$, where $r_c$ is the root of
\begin{equation}\label{rc}
    (1- r)^2 - ( r + \alpha (1 - r)) \cos{r} - r (1 - r) \sin{r} =0 .
\end{equation}
\end{theorem}
\begin{proof}
    Since $f\in \mathcal{S}^*_{nc}$, therefore there exists a Schwarz function $\omega(z)$ satisfying $\omega(0)=0 $ and $\vert \omega(z) \vert \leq \vert z\vert $ such that
    $$ \frac{z f'(z)}{f(z)} = \frac{1 + \omega(z)}{\cos{\omega{(z)}}}, $$
    A simple computation gives
\begin{equation}\label{avehi}
   1 + \frac{z f''(z)}{f'(z)}= \frac{z \omega'(z)}{1+ \omega(z)} + \frac{ z \sin{\omega(z)}}{\cos{\omega(z)}} + \frac{1 + \omega(z)}{ \cos{\omega(z)}}.
\end{equation}
     From (\ref{avehi}), we get
\begin{align*}
    \RE \bigg( 1 + \frac{z f''(z)}{f'(z)} \bigg) &= \RE \bigg( \frac{z \omega'(z)}{1+ \omega(z)}\bigg) + \RE \bigg( \frac{ z \sin{\omega(z)}}{\cos{\omega(z)}} \bigg) +\RE \bigg(  \frac{1 + \omega(z)}{ \cos{\omega(z)}} \bigg) \\
    & \geq  - \frac{\vert z \omega'(z) \vert}{\vert 1+ \omega(z)\vert} + \RE \bigg( \frac{ z \sin{\omega(z)}}{\cos{\omega(z)}} \bigg) +\RE \bigg(  \frac{1 + \omega(z)}{ \cos{\omega(z)}} \bigg).
\end{align*}
  Since $\vert \omega(z) \vert \leq \vert z\vert$. Let $\omega(z) = R e^{i t}$ with $R \leq \vert z \vert = r < 1$ and $t \in [0,2\pi]$.   Using the Schwarz pick lemma
  $$\vert \omega'(z) \vert \leq  \frac{1 - \vert \omega(z) \vert^2}{1- \vert z \vert^2}$$
  and  Theorem \ref{thm1},  we obtain
\begin{align*}
    \RE \bigg( 1 + \frac{z f''(z)}{f'(z)} \bigg) \geq - \frac{r}{1-r} - \frac{r \sin{r}}{\cos{r}} + \frac{1 - r}{\cos{r}}=: h(r).
\end{align*}
   Thus $f\in \mathcal{C}(\alpha)$, whenever $h(r) > \alpha$. Since $h(r)$ is a decreasing function and $h(0)=1$, therefore $h(r)> \alpha$ for $r \in (0, r_c)$, where $r_c$ is the root of (\ref{rc}).
\end{proof}
\begin{corollary}
     The $\mathcal{C}$-radius, for the class $\mathcal{S}^*_{nc}$ is $r_c\approx 0.454. $
\end{corollary}
\begin{theorem}
   If $f\in\mathcal{S}^*_{nc}$ then $f$ is M-starlike in $\vert z\vert <r_3(M)$, where $r_3(M)$ is the root of the equation
     $$1-r-2 M\cos{r}=0, $$
     when $0<M\leq1/2$ and $r_3(M)=1$ when $M> 1/2$.
\end{theorem}
\begin{proof}
    Since $f\in\mathcal{S}^*_{nc}$, there exist a Schwarz function $\omega(z)$ satisfying $\vert \omega(z)\vert <\vert z\vert $ such that
    $$ \frac{z f'(z)}{f(z)} = \frac{1+\omega}{\cos{\omega}}. $$
    Thus,  $f\in\mathcal{S}^*_{nc}$ is M-starlike whenever
    $$\bigg\vert \frac{1+\omega}{\cos{\omega}}-M\bigg\vert <M. $$
    By triangle inequality we can rewrite it as
    $$\bigg\vert \frac{1+\omega}{\cos{\omega}}\bigg\vert <2 M. $$
    And we know that
    $$\frac{1-\vert z\vert }{\cos{\vert z\vert }} \leq  \bigg\vert \frac{1+\omega}{\cos{\omega}}\bigg\vert <2 M.$$
   From the above inequalities we get the desired result.
\end{proof}
\subsection{Subordination Results}
 We need certain notations and results from the theory of differential subordination to establish our main results.
    Let $\mathcal{H}[a,n]$ denote the class of analytic functions $f$ in  $\mathbb{D}$ having the series expansion
   $$ f(z)= a + a_n z^n + a_{n+1} z^{n+1} +  a_{n+2} z^{n+2} + \cdots. $$
  We take $H[1,1]:=\mathcal{H}$ for simplicity. Additionally, let $\mathcal{Q}$ be the class of analytic and univalent functions in $\overline{\mathbb{D}}\setminus E(q)$, where
   $$ E(q) = \{ \zeta \in \partial \mathbb{D} : \lim_{z \rightarrow \zeta} q(z) = \infty \}, $$
   and are such that $q'(\zeta) \neq 0 $ for $\zeta \in \partial \mathbb{D}\setminus E(q)$.
\begin{lemma}\label{lmds1}\cite[p.24]{MiMo} Let $q\in \mathcal{Q}$ with $q(0)=a$, and let $p(z)=a+a_nz^n+\cdots$ be analytic in $\mathbb{D}$ with $p(z)\neq a$ and $n\geq 1$. If $p$ is not subordinate to $q$, then there exist points $z_0=r_0 e^{i\theta_0}\in\mathbb{D}$ and $\zeta_0\in\partial\mathbb{D}\setminus E(q)$ and an $m\geq n\geq 1$ for which $p(\mathbb{D}_{r_0})\subset q(\mathbb{D})$,
\begin{enumerate}
    \item $p(z_0)=q(\zeta_0),$
    \item $z_0p'(z_0)=m\zeta_0 q'(\zeta_0)$,
    \item $\RE \{\frac{z_0 p''(z_0)}{p'(z_0)}+1 \}\geq m\RE \{\frac{\zeta_0 q''(\zeta_0)}{q'(\zeta_0)}+1 \}$
\end{enumerate}
\end{lemma}
\begin{lemma}\cite[Corollary 3.4h, p.132]{MiMo}\label{lemma}
     Let $q$ be univalent in $\mathbb{D}$, and let $\theta$ and $\psi$ be analytic in a domain $D$ containing $q(\mathbb{D})$ with $\psi(w)\neq 0$, when $w \in q(\mathbb{D})$. Set $Q(z) := z q'(z) \psi( q(z) )$ and $h(z):= \theta (q(z)) + Q(z)$. Suppose that either $h$ is convex or $Q(z)$ is starlike. In addition, assume that $\RE (z h'(z)/Q(z))>0$ $(z\in \mathbb{D})$. If $p$ is analytic in $\mathbb{D}$ with $p(0)= q(0)$, $p(\mathbb{D}) \subset D$ and
     $$ \theta (p(z)) + z p'(z) \psi (p(z)) \prec \theta (q(z)) + z q'(z) \psi (q(z)), $$
     then $p \prec q$ and $q$ is the best dominant.
\end{lemma}
\begin{theorem}
 If $p \in\mathcal{H}$ such that $p(0)=1$ and
\begin{equation}\label{hypt1}
    \RE\bigg\{\frac{zp'(z)}{p(z)}\bigg\}<\frac{1}{2} + k_2 ,
\end{equation}
   where $k_2=1/\cosh{2}$, then
\begin{equation*}
    p(z)\prec \frac{1+z}{\cos{z}}.
\end{equation*}
\end{theorem}
\begin{proof}
   Suppose $p(z)\not\prec q(z)= \varphi_{nc}(z)$, then according to  Lemma \ref{lmds1}, there exist points $z_0 = r_0 e^{i \theta_0}\in\mathbb{D}$ and $\zeta_0$ with $\vert \zeta_0\vert = 1$, such that we have $p(\mathbb{D}_{r_0})\subset q(\mathbb{D}),$
  $$p(z_0)=q(\zeta_0) $$
  and
  $$z_0 p'(z_0)=m\zeta_0 q'(\zeta_0). $$
  Thus, we deduce that
  $$\frac{z_0 p'(z_0)}{p(z_0)}=m\frac{\zeta_0 q'(\zeta_0)}{q(\zeta_0)}. $$
  After some simple calculations we can easily get
\begin{align*}
   \RE \left({\frac{z_0 p'(z_0)}{p(z_0)} } \right) = m\RE \left( \frac{\zeta_0 q'(\zeta_0)}{q(\zeta_0)} \right) &  = m\RE \left( \frac{\zeta_0 \varphi_{nc}'(\zeta_0)}{\varphi_{nc}(\zeta_0)} \right) \\
                                                   &=m \bigg(\frac{1}{2}+\frac{1}{\cosh{(2\sin{\theta})}}\bigg) \\
                                                   & \geq m\bigg(\frac{1}{2}+\frac{1}{\cosh{2}}\bigg)
\end{align*}
   for $\theta \in (0,2\pi)$.
  This contradicts the hypothesis (\ref{hypt1}), and therefore, $p(z)\prec q(z).$
\end{proof}
     If we take  $p(z)={zf'(z)}/{f(z)}$ for $ z\in\mathbb{D} $,  then the following result can be readily derived.
\begin{corollary}
      If $f\in\mathcal{A}$ and
\begin{equation*}
    \RE\bigg\{1+\frac{zf''(z)}{f'(z)}-\frac{zf'(z)}{f(z)}\bigg\}<\frac{1}{2} + k_2 \quad (z\in\mathbb{D}),
\end{equation*}
   then $f \in\mathcal{S}^*_{nc}$.
\end{corollary}
    The next theorem provides an additional sufficient condition for a function $f$  to belong to the class $\mathcal{S}^*_{nc}$.
\begin{theorem}
    If $f\in\mathcal{A}$ and satisfies
    $$\RE\bigg(1+\frac{zf''(z)}{f'(z)}\bigg)> \frac{1}{2}+\frac{2+\sinh{1}}{\cos{1}} \quad ( z \in \mathbb{D}), $$
   then $f\in\mathcal{S}^*_{nc}$.
\end{theorem}
\begin{proof}
    Suppose $f\not\in \mathcal{S}^*_{nc}$ or equivlently
    $$\frac{zf'(z)}{f(z)}\not\prec \varphi_{nc}(z). $$
    Then according to Lemma \ref{lmds1}, there are points $z_0= r_0 e^{i \theta_0} \in\mathbb{D}$, $\zeta_0\in\partial\mathbb{D}$ and  $m>1$ such that
    $$\frac{z_0 f'(z_0)}{f(z_0)}=\varphi_{nc} (\zeta_0) \quad \text{and} \quad \bigg[z\bigg(\frac{z f'(z)}{f(z)}\bigg)'\bigg]\bigg\vert _{z=z_0}=m \zeta_0 (\varphi_{nc}(\zeta_0))'. $$
    A computation gives
\begin{align*}
    1+\frac{z_0 f''(z_0)}{f'(z_0)}&=\varphi_{nc}(\zeta_0) + \frac{m\zeta_0 (\varphi_{nc}(\zeta_0))'}{\varphi_{nc}(\zeta_0)}\\
    &=\frac{1+\zeta_0}{\cos{\zeta_0}} + m \bigg(\frac{\zeta_0}{1 + \zeta_0} + \frac{\zeta_0 \sin{\zeta_0}}{\cos{\zeta_0}} \bigg).
\end{align*}
    Thus, we have
\begin{equation}\label{eqn3.4}
    \RE\bigg(1+\frac{z_0 f''(z_0)}{f'(z_0)}\bigg)\leq \bigg\vert \frac{1+\zeta_0}{\cos{\zeta_0}}\bigg\vert +m\bigg(\bigg\vert \frac{\zeta_0}{1+\zeta_0}\bigg\vert +\bigg\vert \frac{\zeta_0 \sin{\zeta_0}}{\cos{\zeta_0}}\bigg\vert \bigg).
\end{equation}
    Now, we find the maximum value of right hand side of the above inequality. Consider
    $$\vert \cos{\zeta}\vert ^2=\vert \cos({e^{i\theta}})\vert ^2=\cos^2({\cos{\theta}})\cosh^2{(\sin{\theta})}+\sin^2{(\cos{\theta})}\sin^2{(\sin{\theta})}=:b(\theta), $$
    where $-\pi\leq\theta\leq\pi$. It is easy to see that the equation $b'(\theta)=0$ has five roots namely $0$, $\pm\pi$ and $\pm\pi/2$ in $[-\pi,\pi]$. Since $b(\theta)=b(-\theta)$, it is sufficient to consider only those roots which lie in $[0,\pi]$. Now
    $$\min\{b(0),b(\pi),b(\pi/2)\}=\cos^2{1}.$$
    Therefore, $\vert \cos{(e^{i\theta}})\vert \geq \cos{1}$. Similarly, it can be verified that
    $$\vert \sin{e^{i\theta}}\vert \leq \sinh{1}.$$
    Using these inequalities in (\ref{eqn3.4}), we obtain
    $$\RE\bigg(1+\frac{z_0 f''(z_0)}{f'(z_0)}\bigg)\leq \frac{2}{\cos{1}}+m\bigg(\frac{1}{2}+\frac{\sinh{1}}{\cos{1}}\bigg), $$
    which contradicts the hypothesis. Consequently, we have
    $$\frac{zf'(z)}{f(z)}\prec \varphi_{nc} (z) $$
    and $f\in\mathcal{S}^*_{nc}$.
\end{proof}
   Now, we consider the first order differential subordination implication of the form
    $$ p(z) + \gamma z p'(z) \prec \frac{1+z}{\cos{z}} \implies p(z) \prec q(z). $$
    We find the range of $\beta \in \mathbb{R}$ such that $p(z) \prec (1 + A z)/(1 + B z)$, $e^z$, $1 + z e^z$ and $1 + \sin{z}$ hold.
\begin{theorem}
     Let $p$ be analytic function in $\mathbb{D}$ such that $p(0)=1$ and satisfies
     $$ 1 + \gamma z p'(z) \prec \frac{1+z}{\cos{z}}. $$
     Let
\begin{equation}\label{fng(z)}
      g(z)= \int_0^z \frac{1 + t - \cos{t}}{t \cos{t}} dt.
\end{equation}
     Then the following hold:
\begin{enumerate}
  \item $ p(z) \prec {(1 + A z)}/{(1 + B z)}$, \text{whenever}
  $$ \gamma \geq \max \bigg\{\frac{\gamma_2 (1- B)}{A - B},  \left( \frac{1 + B^2}{A- B - 1 -B^2} \right)  \IM g(i)\bigg\}, $$
    where $-1 \leq B < A \leq 1$ and $\IM g(i) \approx 0.862897.$
  \item $p(z) \prec e^z $ whenever $\gamma \geq ( e \gamma_1)/(1 - e) \approx 1.4308,$
  \item $p(z) \prec 1 + z e^z $ whenever $\gamma \geq \gamma_2 \approx 2.45796,$
  \item $p(z) \prec 1+ \sin{z}$ whenever $\gamma \geq \gamma_2 \csc{1} \approx 1.82614, $
\end{enumerate}
    where $\gamma_1= g(-1) \approx -0.904233$ and $\gamma_2 =  g(1) \approx 1.53664. $
\end{theorem}
\begin{proof}
\begin{enumerate}
   \item  The differential equation
    $$ 1 + \gamma z q'(z) = \frac{1+z}{\cos{z}} $$
    has a solution $q_\beta : \overline{\mathbb{D}} \rightarrow \mathbb{C}$ given by
\begin{equation*}
     q_\gamma(z) = 1 + \frac{1}{\gamma} \int_0^z \frac{1 + t - \cos{t}}{t \cos{t}} dt= 1 + \frac{1}{\gamma} \bigg( z + \frac{z^2}{4} + \frac{z^3}{6} + \frac{5 z^4}{96} + \frac{z^5 }{24} + \cdots \bigg).
\end{equation*}
    Taking $\theta(w) =1$ and  $\psi(w)=\gamma$ in Lemma \ref{lemma}, the function $Q : \mathbb{D} \rightarrow \mathbb{C}$ becomes
    $$ Q(z) =  z q'_{\gamma}(z) \psi(q_\gamma(z)) = \gamma z q_{\gamma}'(z) = \frac{1 + z - \cos{z}}{\cos{z}}.  $$
    A simple calculation shows that $\RE (z Q'(z)/Q(z))> 0$ for $z\in \mathbb{D}$. Hence, $Q(z)$ is starlike function in $\mathbb{D}$. Since $h(z) =1 + Q(z)$. Therefore $\RE (z h'(z))/Q(z) > 0$. Consequently by Lemma \ref{lemma}, $ 1 + \gamma z p'(z) \prec 1 + \gamma z q'(z)$ implies $p(z) \prec q_\gamma(z)$.
    Now, we only need to show that $q_\gamma(z) \prec \varphi_0(z)=(1+ A z)/(1+ B z)$ for each parts.  If $q_\gamma (z) \prec \varphi_0(z)$, then
\begin{equation}\label{eqjc}
     \varphi_0 (-1) \leq q_\gamma (-1) < q_\gamma (1) \leq \varphi_0 (1)
\end{equation}
    and
\begin{equation}\label{eqjc2}
    \IM q_\gamma(i)  < \IM \varphi_0 (i).
\end{equation}
   A graphical representation shows that the conditions (\ref{eqjc}) and (\ref{eqjc2}) are necessary as well as sufficient for $q_\gamma \prec \varphi_0(z).$ However, in other cases, the condition (\ref{eqjc}) is sufficient.
   Now using (\ref{eqjc}), we obtain
   $$ \frac{1 - A}{1 - B} \leq 1 + \frac{1}{\gamma} g(-1) \;\; \text{which implies} \;\; \gamma \geq \frac{\gamma_1 (1- B)}{B - A} =:\gamma_3 $$
   and
   $$ \frac{1 + A}{1 + B} \leq 1 + \frac{1}{\gamma} g(1) \;\; \text{which implies} \;\; \gamma \geq \frac{\gamma_2 (1- B)}{A - B} =:\gamma_4 . $$
   It can be easily seen that $\gamma_4 \geq \gamma_3$ for the range $-1 \leq B < A \leq 1.$
   From (\ref{eqjc2}), we get
   $$ 1 + \frac{1}{\gamma} g(i) \leq \frac{1 + A i}{ 1 + B i} \;\; \text{which implies} \;\; \gamma \geq \bigg( \frac{1 + B^2}{A- B - 1 -B^2} \bigg)  \IM g(i)=: \gamma_5. $$
   Thus $q_\gamma (z) \prec \varphi_0(z)$ whenever $\gamma \geq \max \{ \gamma_4 ,\gamma_5\}.$
  \item For $\varphi_0 (z) = e^{z}$, (\ref{eqjc}) reduces to
  $$   \frac{1}{e} \leq   q_\gamma (-1) < q_\gamma (1) \leq e. $$
  Now using (\ref{fng(z)}) in the above inequality, we get
  $$   \frac{1}{e} \leq   1 + \frac{1}{\gamma} g(-1) \;\; \text{which implies} \;\; \gamma \geq \frac{g(-1) e}{1 - e} = \frac{ e \gamma_1}{1 - e}  \approx 1.4308 $$
  and
  $$ \frac{1}{\gamma} g(1) \leq e -1 \;\; \text{which implies} \;\; \gamma \geq \frac{g(1)}{e-1} = \frac{\gamma_2}{e-1} \approx 0.8942. $$
  Therefore $q_\gamma (z) \prec e^z$ whenever $\gamma \geq ( e \gamma_1)/(1 - e) \approx 1.4308. $
  \item When $\varphi_0(z) = 1 + z e^z$. Then from (\ref{eqjc}), we get
  $$ 1 - \frac{1}{e} \leq q_\gamma( -1) < q_\gamma (1) \leq 1+ e .$$
  Using (\ref{fng(z)}), we obtain
  $$ 1 - \frac{1}{e} \leq 1 + \frac{1}{\gamma} g(-1) \;\; \text{which gives} \;\; \gamma \geq  - e g(-1) = - e \gamma_1 \approx 2.45796.  $$
  and
  $$ 1 + \frac{g(1)}{\gamma} \leq 1 + e \;\; \text{which gives} \;\; \gamma \geq \frac{g(1)}{e} =  \frac{\gamma_2}{e} \approx 0.56529. $$
  Thus $q_\gamma (z) \prec 1 + z e^z$ whenever $\gamma \geq \gamma_2 \approx 2.45796. $
  \item When $\varphi_0(z) = 1+ \sin{z}$. Then from (\ref{eqjc}), we get
  $$  1 - \sin{1} \leq q_\gamma( -1) < q_\gamma (1) \leq 1+ \sin{1} .$$
  Now using (\ref{fng(z)}), we have
   $$ 1 - \sin{1}  \leq 1 + \frac{1}{\gamma} g(-1) \;\; \text{which gives} \;\; \gamma \geq - \gamma_1 \sin{1} \approx 0.760886 .  $$
  and
   $$ 1 + \frac{g(1)}{\gamma} \leq 1 + \sin{1} \;\; \text{which gives} \;\; \gamma \geq \frac{g(1)}{\sin{1}} =  \frac{\gamma_2}{\sin{1}} \approx 1.82614. $$
   Consequently, $q_\gamma (z) \prec 1 + \sin{z}$ whenever $\gamma \geq \gamma_2 \csc{1} \approx 1.82614. $
\end{enumerate}
\end{proof}

  In the following result, we consider the expression $p(z) + z p'(z)/p(z)$, where $p\in \mathcal{H}$, and the region bounded by the family of parabolas described by the equality $u=v^2/2+(2b+1)/2$, $b<1/2$.  These parabolas are symmetric about the real axis having the vertex at $w=(2b+1)/2$. The family of domains containing point 1 inside and bounded by those parabolas may be characterised as
\begin{equation}
    \Omega_b=\{w:\RE(w-b)>\vert w-1-b\vert \},
\end{equation}
    or equivalently
\begin{equation}
    \Omega_b=\{w=u+iv:2u>v^2+2b+1\}.
\end{equation}
\begin{theorem}
     If $p \in \mathcal{H}$ such that
\begin{equation}
    p(z)+\frac{zp'(z)}{p(z)}\in\Omega_b,
\end{equation}
     then $p(z)\prec \varphi_{nc}(z)$ for $b< b_0\approx-0.005796$.
\end{theorem}
\begin{proof}
     Suppose $p(z) \not\prec\varphi_{nc}(z)$. Then according to Lemma \ref{lmds1}, there exists points $z_0 = r_0 e^{i \theta_0}\in\mathbb{D}$, $\zeta_0\in\partial\mathbf{D}$, $\zeta_0\neq 1$, and $m\geq 1$ such that $p(z_0)=\varphi_{nc}(\zeta_0)$ and $z_0 p'(z_0)=m\zeta_0 \varphi_{nc}'(\zeta_0)$. On putting  the value of $\varphi_{nc}(z)$, we obtain
\begin{equation}
    \frac{z\varphi_{nc}'(z)}{\varphi_{nc}(z)}=\frac{z}{1+z}+\frac{z\sin{z}}{\cos{z}}
\end{equation}
    Note that for $\zeta_0=e^{i\theta}$, $\theta\in [-\pi,\pi]$, we have
\begin{align*}
    p(z_0)=\varphi_{nc}(\zeta_0) =& 2 \bigg( \frac{ (x + 1) \cos{x} \cosh{y} - y  \sinh{y} \sin{x}}{\cos{(2x)} + \cosh{(2 y)}} \bigg) \\
                              & \quad\quad \quad\quad\quad  +  2 i \bigg( \frac{  (x + 1) \sinh{y} \sin{x} + y \cos{x} \cosh {y}}{\cos{(2 x)} +\cosh{(2 y)}} \bigg)
\end{align*}
    and
\begin{align*}
    \frac{z_0 p'(z_0)}{p(z_0)}&=\frac{m\zeta_0 \varphi_{nc}'(\zeta_0)}{\varphi_{nc}(\zeta_0)},\\
                             &=m \bigg( \frac{1}{2} + \frac{x \sin (2 x)-y \sinh (2 y)}{\cos (2 x)+\cosh (2 y)} + i \bigg(\frac{y}{(x+1)^2+y^2}+\frac{y \sin (2 x)+x \sinh (2 y)}{\cos (2 x)+\cosh (2 y)} \bigg)\bigg)
\end{align*}
    where $x=\cos{\theta}$ and  $y=\sin{\theta}$. Let us take $ u:=\RE (p(z_0)+{z_0p'(z_0)}/{p(z_0)} )$ and $v:=\IM (p(z_0) + {z_0 p'(z_0)}/{p(z_0)} )$, then  we have

  $$ u = 2 \bigg( \frac{ (x + 1) \cos{x} \cosh{y} - y  \sinh{y} \sin{x}}{\cos{(2x)} + \cosh{(2 y)}} \bigg) + m \bigg( \frac{1}{2} + \frac{x \sin (2 x)-y \sinh (2 y)}{\cos (2 x)+\cosh (2 y)}  \bigg)$$
   and
$$    v  =2 \bigg( \frac{  (x + 1) \sinh{y} \sin{x} + y \cos{x} \cosh {y}}{\cos{(2 x)} +\cosh{(2 y)}} \bigg) + m \bigg( \frac{y}{(x+1)^2+y^2}+\frac{y \sin (2 x)+x \sinh (2 y)}{\cos (2 x)+\cosh (2 y)} \bigg) .$$
    Observe, by the definition of $\Omega_b$, that will be led to a contradiction if we show that
    $$v^2-2u+2b+1\geq 0. $$
    Observe now that
     $$v^2-2u+2b+1\geq \min\{v^2-2u\}+2b+1. $$
     After some calculation, we can get minimum will be $\approx -0.988408$ at $\theta\approx-2.47734$. Hence
     $$v^2-2u+2b+1\geq 0, $$
     if and only if $-0.988408+2b+1\geq 0$, that is, when $b\geq b_0\approx-0.005796$ that leads to the contradiction. Thus the proof is complete.
\end{proof}
\begin{corollary}
      Let $f\in\mathcal{A}$ and let $1+zf''(z)/f'(z)\in\Omega_b$, where $\Omega_b$ is given in (3.9) and $b\geq b_0\approx-0.005796$. Then
      $$\frac{zf'(z)}{f(z)}\prec\frac{1+z}{\cos{z}}$$
      or equivalently $f\in\mathcal{S}^*_P$.
\end{corollary}
\begin{proof}
      Assume that $1+zf''(z)/f'(z)\in \Omega_b$. Setting $p(z)=zf'(z)/f(z)$, we have $p(0)=1$, and the above condition can be rewritten as
      $$p(z)+\frac{zp'(z)}{p(z)}\in\Omega_b, $$
      in $\mathbb{D}$. Now, applying Theorem 3.5, we conclude the assertion.
\end{proof}
\section{Convolution results}
\begin{theorem}\label{thm24}
 A function $f\in\mathcal{S}^*_{nc}$ if and only if
    $$\frac{1}{z}\bigg[f(z)*\frac{z-\varphi(e^{i\theta})(z-z^2)}{(1-z)^2}\bigg]\neq 0 , \quad \theta\in [-\pi,\pi]$$
\end{theorem}
\begin{proof}
Let $f\in\mathcal{S}^*_{nc}$ then by the definition
   $$\frac{zf'(z)}{f(z)}\prec \frac{1+z}{\cos{z}}. $$
   According to the definition of subordination, there exists a Schwarz function $\omega(z)$ with $\omega(0)=0$ and $\vert \omega(z)\vert <1$, such that
   $$\frac{zf'(z)}{f(z)} =   \frac{1+\omega(z)}{\cos{(\omega(z))}} . $$
   Again by the definition of subordination we can get
   $$\frac{zf'(z)}{f(z)}\neq  \frac{1+ e^{i \theta} }{\cos{e^{i \theta}}},$$
   which is equivalent to
\begin{equation}
  \frac{1}{z}\bigg[ zf'(z)-f(z) \varphi(e^{i\theta})\bigg]\neq 0.
\end{equation}
   Now using the following basic convolution property in (4.3)
\begin{equation}
     f(z)=f(z)*\bigg(\frac{z}{1-z}\bigg) \quad \text{and} \quad zf'(z)=f(z)*\bigg(\frac{z}{(1-z)^2}\bigg)
\end{equation}
    we can get desired result.
\end{proof}
    If we put $f(z)=z+\sum_{n=2}^\infty a_n z^n$ in Theorem \ref{thm24}, we get the condition in terms of coefficients.
\begin{corollary}\label{crl6}
      If $f\in\mathcal{A}$, then a necessary and sufficient condition for $f\in \mathcal{S}^*_{nc}$ is that
      $$\sum_{n=2}^\infty(\varphi(e^{i\theta})-n)  a_n z^{n-1}+\varphi(e^{i\theta})\neq 1,$$
      where $\varphi(z) = (1+z)/\cos{z}$.
\end{corollary}
\begin{corollary}
      If $f(z) = z + \sum_{n=2}^\infty a_n z^n \in \mathcal{A}$ such that
      $$\sum_{n=2}^\infty \vert n-\varphi(e^{i\theta})\vert \vert a_n\vert +M<1, \quad z\in\mathbb{D}$$
      then $f\in\mathcal{S}^*_{nc},$ where $M = 4 \cos{1}/(1 + \cos{2}).$
\end{corollary}
\begin{proof}
      Ler $\varphi(z)= (1+z)/\cos{z}$. Then, from Corollary \ref{crl6}, we have
      $$\bigg\vert 1+\sum_{n=2}^\infty (n-\varphi(e^{i\theta}))a_n z^{n-1}-\varphi(e^{i\theta})\bigg\vert >0.$$
      By triangle inequality, we have
\begin{align*}
    \bigg\vert 1+\sum_{n=2}^\infty (n-\varphi(e^{i\theta}))a_n z^{n-1}-\varphi(e^{i\theta})\bigg\vert &\geq1-\sum_{n=2}^\infty \vert (n-\varphi(e^{i\theta}))\vert  \vert a_n\vert  \vert z\vert ^{n-1}-M,\\
    &>1-\sum_{n=2}^\infty \vert (n-\varphi(e^{i\theta}))\vert  \vert a_n\vert -M
\end{align*}
       $$\bigg\vert 1+\sum_{n=2}^\infty (n-\varphi(e^{i\theta}))a_n z^{n-1}-\varphi(e^{i\theta})\bigg\vert \geq1-\sum_{n=2}^\infty \vert (n-\varphi(e^{i\theta}))\vert  \vert a_n\vert  \vert z\vert ^{n-1}-M,$$
       and
       $$1-\sum_{n=2}^\infty \vert (n-\varphi(e^{i\theta}))\vert  \vert a_n\vert  \vert z\vert ^{n-1}-M>1-\sum_{n=2}^\infty \vert (n-\varphi(e^{i\theta}))\vert  \vert a_n\vert -M.$$
       By corollary \ref{crl6}, above inequality gives the desired result.
\end{proof}

\section*{Declarations}
\subsection*{Funding}
The work of the Surya Giri is supported by University Grant Commission, New Delhi, India,  under UGC-Ref. No. 1112/(CSIR-UGC NET JUNE 2019).
\subsection*{Conflict of interest}
	The authors declare that they have no conflict of interest.
\subsection*{Author Contribution}
    Each author contributed equally to the research and preparation of the manuscript.
\subsection*{Data Availability} Not Applicable.

\end{document}